\documentclass[11pt]{amsart}
\usepackage[margin=1.2in, top=1.2in, bottom=1.2in]{geometry}
\usepackage{amsmath}
\usepackage{amsfonts}
\usepackage{amssymb}
\usepackage{graphicx}
\usepackage{mathrsfs}
\usepackage{graphicx,color}
\usepackage{amsmath,amsfonts,amsthm,enumerate,amscd,latexsym,curves}
\usepackage[mathscr]{eucal}
\usepackage{caption}
\usepackage{subcaption}
\usepackage{color}
\usepackage{comment}
\usepackage{tikz-cd}
\usepackage{graphicx,color}
\usepackage[all]{xy}
\usepackage{mathrsfs}
\usepackage{marvosym}
\usepackage{mathtools,tikz-cd}
\usepackage{graphicx}
\usepackage{enumerate}

\usepackage{tikz-cd}
\usepackage{tikz}
\usetikzlibrary{shapes.geometric, arrows}
\tikzstyle{startstop} = [rectangle, rounded corners, minimum width=2.5cm, minimum height=0.5cm,text centered, text width=2cm, draw=black, fill=white!30]
\tikzstyle{startstop2} = [rectangle, rounded corners, minimum width=2.5cm, minimum height=0.5cm,text centered, text width=2cm, draw=black, fill=white!30]
\tikzstyle{startstop3} = [rectangle, rounded corners, minimum width=2.5cm, minimum height=0.5cm,text centered, text width=2.5cm, draw=black, fill=white!30]
\tikzstyle{arrow} = [thick,->,>=stealth]

\newtheorem{theorem}{Theorem}[section]
\newtheorem{lemma}[theorem]{Lemma}
\newtheorem{question}[theorem]{Question}
\newtheorem{corollary}[theorem]{Corollary}
\newtheorem{proposition}[theorem]{Proposition}

\newtheorem{remark}[theorem]{Remark}

\numberwithin{equation}{section}
\numberwithin{figure}{section}

\newcommand{\CC}{\mathbb{C}}
\newcommand{\CP}{\mathbb{CP}}
\newcommand{\cS}{\mathcal{S}}

\newcommand{\ev}{\mathrm{ev}}

\newcommand{\Symp}{\mathrm{Symp}}
\newcommand{\Diff}{\mathrm{Diff}}

\def\eM{\EuScript{M}}

\newcommand{\PSL}{\mbox{PSL}}

\begin{document}

\title
{Infinite connected components of the space of symplectic forms on ruled surfaces}

\author{Jianfeng Lin and Weiwei Wu}

\maketitle

\begin{abstract}
    We provide an infinite family of diffeomorphic symplectic forms on ruled surfaces, which are pairwise non-isotopic. This answers a uniqueness question regarding symplectic structures up to isotopy on closed symplectic four-manifolds.
\end{abstract}

\section{Introduction}
A fundamental question in symplectic topology is to study the existence and uniqness of symplectic structures on a given smooth manifold.  There is a zoo of uniqueness problems of symplectic structures with many different flavors.  We recommend Salamon's beautiful survey \cite{SalamonSurvey} to interested readers, as well as related sections in \cite{McDuffSalamon}. There are many symplectic manifolds, mostly of dimension 4, which are known to have a unique symplectic structure up to diffeomorphisms, e.g. $\CP^2$ and symplectic ruled surfaces.

In this note, we also focus on dimension $4$.  Let $X$ be a closed 4-manifold and let $a\in H^{2}(X;\mathbb{R})$ be a class. Consider the space of symplectic forms
\[
\mathcal{S}_{a}:=\{\text{symplectic form $\omega$ on $X$ with $[\omega]=a$}\}
\]
Any path between two forms in $\mathcal{S}_a$ is called an \textit{isotopy} of symplectic forms.  We will consider the following well-known open question (see \cite{SalamonSurvey}, \cite[13.1]{MSIntro}).
\begin{question}\label{question isotopy} Does there exist a closed 4-manifold $X$ such that $\mathcal{S}_{a}$ is disconnected? 
\end{question}

 Question \ref{question isotopy} is equivalent to asking whether cohomologous symplectic forms are unique up to isotopy.  It is known to McDuff \cite{McEx} that $\cS_a$ can be disconnected for some six-dimensional manifolds.  

For the simplest cases of symplectic 4-manifolds such as complex projective planes and quadric surfaces, one may boil down the question to asking for the connectedness of diffeomorphism groups of underlying smooth manifolds \cite{SalamonSurvey}, thanks to the series of Gromov \cite{Gromov85}, Taubes \cite{TaubesGr,TaubesGrSW,TaubesSW}, McDuff \cite{McDuffRuled}, Abreu-McDuff \cite{Abreu}.  
For non-compact manifolds, Wang \cite{wang2024note} showed that there are infinitely pairwise distinct non-standard symplectic forms that are standard at infinity on $S^1\times D^3$.

The main result of this paper gives an affirmative answer to Question \ref{question isotopy} by showing the space of cohomologous symplectic forms on an irrational ruled surfaces has infinitely many components.

\begin{theorem}\label{thm: main} Let $S^2\to X\xrightarrow{\pi}\Sigma$ be a ruled surface with $g(\Sigma)>0$. Then for any $a$ with $a^2>0$, the space $\mathcal{S}_a$ has infinitely many components. 
\end{theorem}

\begin{remark} 
Symplectic forms on a ruled surface $X$ have been completely classified up to diffeomorphisms by Lalonde-McDuff \cite{LalondeMcDuff}. In particular, $\mathcal{S}_{a}$ is nonempty if and only if $a^2> 0$. And they further proved that any two cohomologous symplectic forms on $X$ are diffeomorphic. 
Consider the space $\mathcal{S}=\cup_{a}\mathcal{S}_{a}$ of \emph{all} symplectic forms on $X$. We say two symplectic forms are homotopic if they are in the same path component of $\mathcal{S}$. \footnote{Many recent literatures refer to such forms as ``deformation equivalent" \cite{McD96}, but the same terminology in \cite{SalamonSurvey} allows a possible composition by diffeomorphisms.  Our current convention tries to avoid possible confusion.  }
By \cite{LalondeMcDuff,LalondeMcDuffJ}, two cohomologous symplectic forms on ruled surfaces are isotopic if and only if they are homotopic. Therefore, Theorem \ref{thm: main} implies that $\mathcal{S}$ also has infinitely many components. This gives the first example of non-homotopic symplectic forms on closed 4-manifolds with identical Chern classes. It also gives the first example of symplectic forms closed 4-manifolds that are formally homotopic (i.e., connected by a path of nondegenerate forms) but not homotopic. See \cite{CieliebakHprinciple}. In higher dimensions, such examples were established by Ruan \cite{Ruan94} using the Gromov-Witten invariants. On the other hand, by the work of Taubes, the Gromov-Witten invariants of symplectic 4-manifolds only depend on the smooth structure \cite{TaubesGr,TaubesGrSW,TaubesSW}.
\end{remark}

To prove Theorem \ref{thm: main}, we take the product symplectic form on a ruled surface and pull it back by the barbell diffeomorphisms (introduced by Budney-Gabai \cite{budney2019knotted}). This construction is similar to McDuff's example of non-isotopic symplectic forms on 6-manifolds \cite{McEx}, obtained by pulling back the product symplectic form on $T^2\times S^2\times S^2$ by a diffeomorphism. A key difference is that McDuff used a diffeomorphism that respects the ruling, while our diffeomorphism changes the ruling in an essential way. Indeed, we use the Dax invariant to prove that the pulled-back ruling is not smoothly isotopic to the original one. On the symplectic side, we prove that two rulings must be smoothly isotopic if they are compatible with two isotopic symplectic structures. Combining these results, we show that the pulled-back symplectic forms are not isotopic to each other.



\subsection*{Acknowledgement} J. Lin is supported by NSFC Grant 12271281, W. Wu is supported by NSFC 12471063.  The authors are grateful to R. Hind and J. Li for discussions on the isotopy uniqueness problem of symplectic surfaces.

\section{Some isotopy results in symplectic geometry}

 Let $\Sigma$ be a closed, oriented surface with $g(\Sigma)>0$. Let $p: L\to \Sigma$ be a complex line bundle. We assume $L$ is either trivial or $\langle c_{1}(L),[\Sigma]\rangle=1$. We use $X$ to denote the fiberwise one-point compactification of $L$. Then $X$ is the product manifold $S^2\times \Sigma$ if $L$ is trivial; when $\langle c_{1}(L),[\Sigma]\rangle=1$, $X$ is diffeomorphic to a twisted $S^2$-bundle over $\Sigma$, denoted by $S^2\widetilde{\times}\Sigma$. 
 We use $\Sigma_0\hookrightarrow X$ and $\Sigma_\infty\to X$ to denote the section in $0$ and $\infty$. For any $b\in \Sigma$, we use $F_{b}$ to denote the fiber over $b$. We fix a base point $b_0\in \Sigma$ and use $F$ to denote $F_{b_0}$. The space $X$ has a complex structure $J_0$ such that $\Sigma_0$ and $\{F_{b}\}_{b\in \Sigma}$ are all $J_0$-holomorphic. For any $\mu>0$, there exists a symplectic form $\omega_{\mu}$ on $X$ such that $J_0$ is $\omega_\mu$-compatible and that 
 \[
 \operatorname{PD}[\omega_{\mu}]=[\Sigma_0]+\mu [F].
 \]
 Moreover, such symplectic forms are unique up to isotopy. (See \cite{McDuffRuled}.) And any symplectic form on $X$ is diffeomorphic to $\omega_{\mu}$ to rescaling by a constant. \cite{LalondeMcDuff,LalondeMcDuffJ}.


We consider the symplectomorphism group $\Symp(X,\omega_{\mu})$ and the orientation-preserving diffeomorphism group $\Diff^{+}(X)$. We also consider the subgroup $\Diff^{+}_{F}(X)\subset \Diff^{+}(X)$ of fiber-preserving diffeomorphisms. Namely, $\operatorname{Diff}^{+}_{F}(X)$ consists of diffeomorphisms that cover orientation-preserving diffeomorphisms on $\Sigma$. We consider the maps 
\[
i_{1}: \pi_{0}(\Symp(X,\omega_{\mu}))\to \pi_0(\Diff(X))\]
and
\[ i_{2}: \pi_{0}(\Diff_{F}(X))\to \pi_0(\Diff(X))
\]
induced by natural inclusions. 

We start by investigating the group $\operatorname{Diff}^{+}_{F}(X)$. By definition, we have a group homomorphism 
\begin{equation}\label{eq: rho}
\rho: \Diff^{+}_{F}(X)\to \Diff^{+}(\Sigma)    
\end{equation}
such that $f(F_{b})\subset F_{\rho(f)(b)}$ for any $f\in \Diff^{+}(X)$ and any $b\in \Sigma$.

\begin{lemma}\label{lem: rho serre fibration} $\rho$ is a surjective Serre fibration.
\end{lemma}
\begin{proof} 
We first prove that $\rho$ is surjective. For $g\in \Diff^{+}(\Sigma)$, the line bundle $g\circ p: L\to \Sigma$ is isomorphic to the line bundle $p:L\to \Sigma$ because they have identical Chern classes. Therefore, there exists a
diffeomorphism $f':L\to L$ that covers $g$.  Furthermore, by choosing a Hermitian metric on $L$, we may assume $f'$ to be a unitary bundle isomorphism. After fiberwise compactification, we obtain $f\in\Diff_{F}(X)$ with $\rho(f)=g$. 

Next, we prove that $\rho$ satisfies the path lifting property.
Given $f\in \Diff_{F}(X)$ and a path $\gamma:I\to \Diff^{+}(\Sigma)$ with $\gamma(0)=\rho(f)$. We can find a time-dependent vector field $\mathcal{V}$ such that 
\[
\gamma(t)=\varphi_{t}\circ \gamma(0)\in \Diff^{+}(\Sigma),
\]
where $\varphi_{t}:\Sigma\to \Sigma$ is the time-$t$ flow generated by $\mathcal{V}$. We can pick a unitary connection $A$ on $L\to \Sigma$ and use it to lift $\mathcal{V}$ to a vector field $\widehat{\mathcal{V}}$ on $L$. Note that $X=L\cup \Sigma_{\infty}$.
We define a vector field $\widetilde{\mathcal{V}}$ by gluing together $\widehat{\mathcal{V}}$ on $L$ and $\mathcal{V}$ on $\Sigma_{\infty}$. Let $\widetilde{\varphi}_{t}: X\to X$ be the flow generated by $\widetilde{\mathcal{V}}$. Then the path $\widetilde{\gamma}:I\to \Diff^{+}(X)$ defined by $\widetilde{\gamma}(t)=\widetilde{\varphi}_{t}\circ f$ is a lift of $\gamma$. This verifies the path lifting property. 

A general lifting property can be proved similarly.
\end{proof}
By Lemma \ref{lem: rho serre fibration}, we have a long exact sequence of groups
\begin{equation}\label{eq: exact sequence for rho}
\begin{split}
\cdots \to \pi_{1}(\Diff^{+}_{F}(X),\operatorname{id})\rightarrow &\pi_{1}(\Diff^{+}(\Sigma),\operatorname{id})\xrightarrow{\partial}\\ &\pi_{0}(\ker \rho)\xrightarrow{i_*}  \pi_{0}(\Diff^{+}_{F}(X))\xrightarrow{\rho_*} \pi_{0}(\Diff^{+}(\Sigma))\to 0.     
\end{split}
\end{equation}
\begin{lemma}\label{lem: ker rho} There is a canonical isomorphism $\pi_{0}(\ker \rho)\cong H^{1}(\Sigma;\mathbb{Z}/2)$. 
\end{lemma}
\begin{proof}
Given any $\beta\in H^{1}(\Sigma;\mathbb{Z})$, there exists a map $g_{\beta}:\Sigma\to S^1$ such that \[g^*_{\beta}(1)=\beta\in H^{1}(X;\mathbb{Z}).\]
Such $g_{\beta}$ is unique up to homotopy. Consider the map $h_{\beta}: L\to L$ defined by \[h_{\beta}(z)=g_{\beta(b)}\cdot z,\quad  \forall z\in F_{b}.\]
Then we can define the diffeomorphism 
\begin{equation}\label{eq: f-beta}
    f_{\beta}=h_{\beta}\cup \operatorname{id}_{\Sigma_{\infty}}: X\to X. 
\end{equation}

This gives a group homomorphism
\[q_1: H^{1}(\Sigma;\mathbb{Z})\to \pi_0(\ker \rho),\quad q_1(\beta)=[f_{\beta}].\] 

Let $P\to \Sigma$ be the fiber bundle whose fiber over $b\in \Sigma$ is the diffeomorphism group $\Diff^{+}(F_{b})$. Then $\ker\rho$ consists of the diffeomorphisms of $X$, which preserve all fibers, hence it is homotopy equivalent to the space of smooth sections of $P$.  Denote the space of these sections by $\Gamma(P)$. Since the inclusion $SO(3)\hookrightarrow \Diff^{+}(S^2)$ is a homotopy equivalence \cite{Smale}, we have 
\begin{equation}\label{eq: diffS2}
\pi_{0}(\Diff^{+}(S^2))\cong 0,\     \pi_{1}(\Diff^{+}(S^2))\ \text{and }\pi_{2}(\Diff^{+}(S^2))\cong 0. 
\end{equation}
We fix a trivialization of $L|_{\operatorname{sk}_{1}\Sigma}$, which in turn induces a trivialization of $P|_{\operatorname{sk}_{1}\Sigma}\to \operatorname{sk}_{1}\Sigma$. Then for any $s\in \Gamma(P)$, the restriction $s|_{\operatorname{sk}_{1}(\Sigma)}$ can be expressed as a map from $\operatorname{sk}_{1}(\Sigma)$ to $\Diff^{+}(S^2)$. Therefore, any $s\in \Gamma(P)$ gives an element 
\[
[s|_{\operatorname{sk}_{1}(\Sigma)}]\in [\operatorname{Sk}_{1}(\Sigma),\Diff^{+}(S^2)]\cong \operatorname{Hom}(\pi_{1}(\Sigma),\mathbb{Z}/2)\cong H^{1}(\Sigma;\mathbb{Z}/2)
\]
We define the map 
\[
q_2: \pi_{0}(\ker\rho)\cong \pi_{0}(\Gamma(P))\to H^{1}(\Sigma;\mathbb{Z}/2)
\]
by $q_2([s])=[s|_{\operatorname{sk}_{1}(\Sigma)}]$. By (\ref{eq: diffS2}), two sections of $P$ are homotopic if and only if their restrictions to $\operatorname{sk}_{1}\Sigma$ are homotopic. Hence, the map $q_2$ is injective. 

Now we consider the composition $q_2\circ q_1: H^{1}(\Sigma;\mathbb{Z})\to \pi_0(\ker\rho)\to H^{1}(\Sigma;\mathbb{Z}/2)$.
For $\beta\in H^{1}(\Sigma;\mathbb{Z})$, let $s_{\beta}\in \Gamma(P)$ be the section that corresponds to $f_{\beta}$. Then $s_{\beta}|_{\operatorname{sk}_{1}\Sigma}$ is given by the composition
\[
\operatorname{sk}_{1}\Sigma\xrightarrow{g_{\beta}} S^{1}\hookrightarrow SO(3)\hookrightarrow \Diff^{+}(S^2).
\]
Therefore, $q_2\circ q_1$ is just the change-of-coefficient map. In particular, $q_1$ descends to a canonical isomorphism 
 $H^{1}(\Sigma;\mathbb{Z}/2)\xrightarrow{\cong} \pi_{0}(\ker\rho)$.
\end{proof}
\begin{proposition}\label{prop: MCGF}
The following results hold.
\begin{enumerate}
    \item If $X=S^2\times \Sigma$, then  $\pi_{0}(\Diff^{+}_{F}(X))\cong  \pi_0(\Diff^{+}(\Sigma))\ltimes H^{1}(\Sigma;\mathbb{Z}/2).$
    \item If $X=S^2\widetilde{\times}T^2$, then $\pi_{0}(\Diff^{+}_{F}(X))\cong \operatorname{SL}(2,\mathbb{Z})$.
\item If $X=S^2\widetilde{\times} \Sigma$ with $g(\Sigma)>1$, then we have a short exact sequence 
\begin{equation}\label{eq: short exact sequence}
0\to H^{1}(\Sigma,\mathbb{Z}/2)\to \pi_{0}(\Diff^{+}_{F}(X))\to \pi_0(\Diff^{+}(\Sigma))\to 0.    
\end{equation}
\end{enumerate}
\end{proposition}
\begin{proof} (1) Suppose $X=S^2\times\Sigma$, then the fibration (\ref{eq: rho}) has a section $s: \Diff^{+}(\Sigma)\to \Diff_{F}^{+}(X)$ defined by $s(g)= \operatorname{id}\times g$. Therefore, the long exact sequence (\ref{eq: exact sequence for rho}) turns into a split short exact sequence 
\[
0\rightarrow H^{1}(\Sigma;\mathbb{Z}/2)\to \pi_{0}(\Diff^{+}_{F}(X))\to \pi_{0}(\Diff^{+}(\Sigma))\to 0.
\]

(2) It suffices to prove that the map 
\[\partial :\pi_{1}(\Diff^{+}(T^2),\operatorname{id})\to \pi_{0}(\ker\rho)\] is surjective when $X=S^2\widetilde{\times }T^2$. 

For each $\theta\in S^1=\mathbb{R}/\mathbb{Z}$, we consider the loop $\gamma_{\theta}: S^1\to T^2 $ defined by $\gamma_{\theta}(\eta)=(\theta,\eta)$. We pick an $U(1)$-connection  $A$ on $L$ that satisfies the following two conditions:
\begin{enumerate}
    \item The curvature satisfies $\frac{i}{2\pi}F_{A}=d\theta\wedge d\eta$.
    \item The monodromy of $A$ along the loop $\gamma_{0}$ is trivial. 
\end{enumerate}
Such $A$ can be obtained by picking any connection $A_0$ and modifying it with a suitable 1-form $a\in i\Omega^{1}(T^2)$.

Consider the vector field $\mathcal{V}=\frac{\partial}{\partial \eta}$ on $T^2$. Let $\varphi_{t}: \Sigma\to \Sigma$ be the flow generated by $\mathcal{V}$. Then 
\begin{equation}\label{loop phi}
I\to \Diff^{+}(T^2),\ t\mapsto \varphi_{t}    
\end{equation}
is a loop in $\Diff^{+}(T^2)$. As in the proof of Lemma \ref{lem: rho serre fibration}, we can use $A$ to lift $\mathcal{V}$ to a vector field $\widehat{\mathcal{V}}$ on $L$ and then extend it to a vector field $\widetilde{\mathcal{V}}$ on $X$. Let $\widetilde{\varphi}_{t}: X\to X$ be the flow generated by $\widetilde{\mathcal{V}}$. Then the path 
\[
I\to \Diff^{+}(X), \quad t\mapsto \widetilde{\varphi}_{t}
\]
is a lift of the loop (\ref{loop phi}). Let $f\in \widetilde{\varphi}_{1}$. Then $[f]$ belongs to the image of $\partial$.  And $f$ is defined by parallel transportations along the loops $\{\gamma_{\theta}\}$. By our choice of $A$, the monodromy of $A$ along $\gamma_{\theta}$ equals $e^{2\pi i\theta}$. 
Therefore, we have $f(z)=e^{i \theta }z$ for all $z\in F_{(\theta,\eta)}$. As a result, under the isomorphism $\pi_0(\ker\rho )\to H^{1}(T^2;\mathbb{Z}/2)$ provided by Lemma \ref{lem: ker rho}, the class $[f]$ corresponds to the generator $(1,0)$. Similarly, the other generator $(0,1)$ also belongs to the image of $\partial$. 

(3) Note that each component of $\Diff^{+}(\Sigma)$ is contractible when $g(\Sigma)>1$. So $\pi_{1}(\Diff^{+}(\Sigma))=0$. The short exact sequence (\ref{eq: short exact sequence}) follows from (\ref{eq: exact sequence for rho}). 
\end{proof}

\begin{proposition} The map $i_{2}$ is injective. 
\end{proposition}
\begin{proof} Suppose $[f]\in \pi_{0}(\Diff^{+}_{F}(X))$ belongs to $\ker i_2$. Then the composition 
\[
\Sigma\xrightarrow{\sim}\Sigma_0\hookrightarrow X\xrightarrow{ f}X\xrightarrow{p} \Sigma
\]
is homotopic to $\operatorname{id}_{\Sigma}$, where the first map is taken as a section. On the other hand, this composition is exactly $\rho(f)$. Hence $\rho(f)$ is homotopic to $\operatorname{id}_{\Sigma}$. For diffeomorphisms on surfaces, homotopy implies isotopy, so $\rho(f)$ is isotopic to $\operatorname{id}_{\Sigma}$. By (\ref{eq: exact sequence for rho}), $[f]$ belongs to the image of $i_*$. 

Suppose  $X=S^2\widetilde{\times }T^2$. Then by Proposition \ref{prop: MCGF}, the map $i_*$ is trivial and the proof is finished. 

Suppose $X\neq S^2\widetilde{\times }T^2$. Since elements in $\ker(\rho_*)$ can be represented by diffeomorphisms acquired by \eqref{eq: f-beta}, we may assume $f=f_{\beta}$ for some $\beta\in H^{1}(\Sigma;\mathbb{Z})$ up to isotopy. We may also assume that $\beta\neq 0 \in H^{1}(\Sigma;\mathbb{Z}/2)$ because $[f]=0$ otherwise. Take a simple closed curve $\gamma\subset \Sigma_0$ such that
$\langle\beta, \gamma\rangle$ is odd. Consider the mapping torus \[Tf=(X\times [0,1])/(0,x)\sim (1,f(x)).\] 
Since $f$ is constructed as a fiberwise multiplication, it fixes $\gamma$ pointwisely. We therefore have an embedded torus 
\[T:=\gamma\times S^1\hookrightarrow Tf.\] 
The normal bundle of $T$, denoted by $NT$, is isomorphic the mapping torus
\[
(N\gamma\times [0,1])/(0,v)\sim (1,f_{*}(v)).
\]
Here $N\gamma$ is the normal bundle of $\gamma$ in $X$ and $f_*:N\gamma\to N\gamma$ is the differential of $f$. 
Since $\langle \beta,[\gamma]\rangle$ is odd, an explicit calculation shows that that $w_{2}(NT)\neq 0$.

We claim that that for any embedded torus $T'\hookrightarrow X\times S^1$, we have $w_{2}(NT')=0$. This would finish the proof because $Tf$ is not diffeomorphic to $X\times S^1$ and hence $[f]\notin \ker i_2$. 

Now we prove the claim. 
We first assume $X=\Sigma\times S^2$. Then $X\times S^1$ is spin. So we have \[w_{2}(NT')=w_{2}(T(X\times S^1)|_{T'})=0.\] 
Next, we assume $X=\Sigma\widetilde{\times }S^2$ with $g(\Sigma)>1$.  Consider the projection maps 
\[
X\times S^1\xrightarrow{p_1}X\xrightarrow{p}\Sigma.
\]
Since $TS^{1}$ is trivial, we have 
$w_{2}(T(X\times S^1))=p^*_{1}(w_{2}(TX))$. Consider the short exact sequence 
\[
0\to H^{2}(\Sigma)\xrightarrow{p^*}H^{2}(X)\xrightarrow{i^*} H^2(F)\to 0,
\]
where $i:F\to X$ is the inclusion of a fiber. Then $i^*(w_{2}(TX))=w_{2}(TF)=0$. So $\omega_{2}(TX)$ is pulled back from $\Sigma$. This further implies that $w_{2}(T(X\times S^1))$ is pulled back from the projection map $p\circ p_1: X\times S^1\to \Sigma$.

Therefore, $w_{2}(NT')=w_{2}(T(X\times S^1)|_{T'})$ is pulled back from the map 
\[
T'\hookrightarrow X\times S^1\xrightarrow{ \operatorname{p_1}} X\xrightarrow{p}\Sigma.
\]
But this map must have mapping degree zero because $g(\Sigma)>g(T')$. Hence $w_{2}(NT')=0$. This finishes the proof.
\end{proof}

\begin{proposition}\label{prop: smooth isotopy of symplectormorphisms} The image of $i_1$ equals the image of $i_2$.
\end{proposition}
\begin{proof} We first prove that $\operatorname{im}(i_{2})\subset \operatorname{im}(i_{1})$.  This amounts to showing that any $f\in \Diff^{+}_{F}(X)$ is smoothly isotopic to a symplectomorphism. Let $\omega=f^{*}(\omega_{\mu})$. Then $[\omega]=[\omega_{\mu}]$ and they are compatible with the fibration structure on $X$.  By a result of Lalonde-McDuff \cite{LalondeMcDuffJ}, there exists a path of cohomologous symplectic forms that connects  $\omega$ with $\omega_{\mu}$. Applying Moser's argument, this gives a diffeomorphism $g: M\to M$ that is smoothly isotopic to the identity and satisfies $g^{*}(\omega)=\omega_{\mu}$. Then $g\circ f\in \Symp(X,\omega_\mu)$ is smoothly isotopic to $f$. 

Next, we prove that $\operatorname{im}(i_{1})\subset \operatorname{im}(i_{2})$. This amounts to showing that any symplectomophism is smoothly isotopic to a fiber preserving diffeomorphism. Let $f:(X,\omega_{\mu})\to (X,\omega_{\mu})$ be a symplectomorphism. Take a product almost complex structure $J_0$, and denote $J_t$ as a path of almost complex structure connecting $J_0$ and $J_1=f_*(J_0)$.  Define the moduli space

\begin{equation}
    \eM_I:=\{(u,J_t)|u: S^2\to X, \partial_{J_t}u=0, [u]=[S^2]\}.
\end{equation}

Given $x\in X$, there is a unique regular embedded $J_t$-holomorphic sphere passing through $x$ for any $t\in [0,1]$ \cite{LalondeMcDuff}.  Therefore, we have the following fibration

\begin{equation}
    \PSL(2,\CC)\to \eM_I\xrightarrow{\mathfrak{f}} \Sigma\times I.
\end{equation}

Here $\mathfrak{f}$ is the quotient by the natural $\PSL(2,\CC)$-action, which is equivalent to forgetting the parametrization of the curve $S^2$.

Consider $\eM_S:=\eM_I\times_{\PSL(2,\CC)}S^2$.  There is a natural diffeomorphism $\varphi: \eM_S\rightarrow X\times I$ which respects the projection to $I$.  Consider also the evaluation map

\begin{align}
  \ev_I:\quad &\eM_S\to X,\\
         &((u,t),x)\mapsto u(x).
\end{align}

From the fact that the $[S^2]$-curves form a foliation of $X$ for every $J_t$, $\ev_I$ is also a diffeomorphism which respects the projection to $I$.  Therefore, we consider the following composition:
\begin{equation}
    X\times I\xrightarrow{\varphi^{-1}}\eM_S\xrightarrow{\ev_I}X.
\end{equation}
The composition $\ev_I\circ\varphi^{-1}$ yields a smooth isotopy from $f_0:=\ev_I\circ\varphi^{-1}|_{X\times \{0\}}$ to $f_1:\ev_I\circ\varphi^{-1}|_{X\times \{1\}}$, where $f_0$ preserves the standard foliation given by the standard product structure $X\cong S^2\times\Sigma$, while $f_1$ sends the standard foliation to the foliation given by $J_1$-curves.  

 Therefore, $(\ev_I\circ\varphi^{-1})\circ (f_0^{-1}\times id_I)$ is a smooth isotopy from $id_X$ to $f_0^{-1}\circ f_1$, and the latter can be written as $f\circ \alpha$ for some $\alpha\in\Diff^{+}_F(X)$.  This concludes our proof.
\end{proof}

\begin{corollary} Let $X$ be a ruled surface over $T^2$. Then we have 
\[
\pi_0(\Symp(X,\omega_{\mu}))\cong \begin{cases} \operatorname{SL}(2;\mathbb{Z})\ltimes (\mathbb{Z}/2\oplus \mathbb{Z}/2), &\text{if } X=S^2\times T^2;\\
\operatorname{SL}(2;\mathbb{Z}), &\text {if } X=S^2\widetilde{\times}T^2.
\end{cases}
\]
\end{corollary}
\begin{proof} A result McDuff \cite{McDuffALC} (when $X=S^2\times T^2$) and the result of Shevchishin-Smirnov \cite{ShevchishinElliptic} (when $X=S^2\widetilde{\times}T^2$) states that $i_{1}$ is injective. Therefore, $\pi_0(\Symp(X,\omega_{\mu}))\cong \pi_{0}(\Diff^{+}_{F}(X))$.
\end{proof}
\begin{remark}
The symplectic isotopy conjecture for ruled surface \cite{McDuffALC} asserts that the map $i_{1}$ is injective for any ruled surface. This conjecture is still open.    
\end{remark}

The following proposition is independently known to Hind-Li \cite{HLpreprint}. 

\begin{proposition}\label{proposition symplectic suface standard}
Let $\Sigma_1, \Sigma_2$ be embedded, connected symplectic surfaces in $(X,\omega_{\mu})$. Assume 
\[
[\Sigma_1]=[\Sigma_2]=[\Sigma_0]+k[F]
\]
for some $k\in \mathbb{Z}$. Then $\Sigma_1$ is smoothly isotopic to $\Sigma_2$. 
\end{proposition}

\begin{proof}
By the adjunction formula, we have $g(\Sigma_{1})=g(\Sigma_{2})=g(\Sigma)$.
By the symplectic neighborhood theorem, there exist $\omega_{\mu}$-compatible almost complex structure $J_{1}, J_2$ on $X$ such that $\Sigma_{1}$ is $J_{1}$-holomorphic and $\Sigma_2$ is $J_2$-holomorphic. We connect $J_1,J_2$ by  a path $\{J_{t}\}_{1\leq t\leq 2}$ of $\omega_{\mu}$-compatible almost complex structures. Then, as in the proof of  Proposition \ref{prop: smooth isotopy of symplectormorphisms}, we have a smooth fiber bundle 
\[
S^2\hookrightarrow X\times [1,2]\xrightarrow{\pi} \Sigma\times [1,2]
\]
For $t\in [1,2]$, we use $\pi_{t}:X\to \Sigma$ to denote the bundle obtained by restricting $\pi$ to $X\times \{t\}$. By positivity of intersections for $J$-holomorphic curves, the surface $\Sigma_{1}$ (resp. $\Sigma_2$) is a smooth section of the bundle $\pi_{1}$ (resp. $\pi_{2}$).

Now we consider the composition  
\begin{equation}\label{eq: projection}
\Sigma\times [1,2]\xrightarrow{\operatorname{projection}} \Sigma\times \{1\}\hookrightarrow \Sigma\times [1,2].    
\end{equation}
If we pull back the bundle $\pi$ using this map, we will get the bundle 
\[
S^2\hookrightarrow X\times [1,2]\xhookrightarrow{\pi_{1}\times \operatorname{id}} \Sigma\times [1,2].
\]
On the other hand, the map (\ref{eq: projection})  is homotopic to $\operatorname{id}_{\Sigma\times \{1,2\}}$. By homotopy invariance of pull-back bundles, the bundles $\pi$ and $\pi_{1}\times \operatorname{id}$ are isomorphic. In other words, there exists a diffeomorphism $h: X\times [1,2]\to X\times [1,2]$ that fits into the commutative diagram 
\[
\xymatrix
{X\times [1,2]\ar[rd]_{\pi_{1}\times \operatorname{id}}\ar[rr]^{h}_{\cong}& & X\times [1,2]\ar[ld]^{\pi}\\
& \Sigma\times [1,2] &}
\]
By replacing $h$ with $h\circ( h|_{X\times \{1\}}\times \operatorname{id})^{-1}$, we may assume $h$ restricts to the identity map on $X\times \{1\}$. Let $f:X\to X$ be defined by restricting $h$ to $X\times \{2\}$. Then $f$ is isotopic to $\operatorname{id}_{X}$ and fits into the commutative diagram 
\[
\xymatrix
{X\ar[rd]_{\pi_{1}}\ar[rr]^{f}_{\cong}& & X\ar[ld]^{\pi_{2}}\\
& \Sigma&}
\] 
Therefore, $f(\Sigma_{1})$ and $\Sigma_2$ are both smooth sections of the bundle $X\xrightarrow{\pi_{2}}\Sigma$. Since $f(\Sigma_{1})$ is homologous to $\Sigma_{2}$, and any two homologous sections of the same bundle $X\to \Sigma$ are smoothly isotopic, we see that $f(\Sigma_1)$ is smoothly isotopic to $\Sigma_{2}$. Hence $\Sigma_1$ is smoothly isotopic to $\Sigma_2$.
 \end{proof}

\section{The construction of exotic symplectic forms}

We recall the barbell diffeomorphism \[\varphi: (S^2\times D^2)\natural (S^2\times D^2)\to (S^2\times D^2)\natural (S^2\times D^2)\] defined by Budney-Gabai \cite{budney2019knotted}. Fix a small $\epsilon>0$ and a bump function $\rho: [-1,1]\to [0,\epsilon]$ that equals $1$ near $0$ and is supported on $[-\epsilon,\epsilon]$. For any $\theta\in \mathbb{R}/2\pi \mathbb{Z}$, we consider the embedding 
\[i^{\theta}=i^{\theta}_{1}\cup i_{2}: [-1,1]\sqcup [-1,1]\hookrightarrow D^4\] defined by 
\[
i^{\theta}_{1}(t)=(t,0,\cos\theta\cdot \tau(t), \sin\theta\cdot \tau(t))\text{ and }i_{2}(t)=(0,t,0,0).
\]
One can obtain a trivialization of the normal bundle of $i^{\theta}_{1}$ (resp. $i_{2}$) by projecting it to the plane $x=0$ (resp. $y=0$). This makes the family $\{i^{\theta}\}_{\theta\in S^1}$ a loop in $\operatorname{Emb}^{\operatorname{fr}}_{\partial}(I\sqcup I,D^4)$, the space of framed embeddings $I\sqcup I\to D^4$ fixed near the boundary. Applying the isotopy extension theorem to this loop, one obtains a diffeomorphism $\phi: D^4\to D^4$ that fixes a tubular neighborhood of $\nu(\operatorname{im}(i^{0}))$. By restricting $\phi$ to the complement \[D^{4}\setminus \nu(\operatorname{im}(i^{0}))\cong (S^{2}\times D^2)\natural (S^{2}\times D^2),\] one obtains the barbell diffeomorphism $\varphi$. 

Now we construct an embedding $(S^{2}\times D^2)\natural (S^{2}\times D^2)\hookrightarrow X$ as follows: Let $S_{0}=S^2\hookrightarrow  X$ be a fiber. Let $b=S_0\cap \Sigma_0$. Take a small embedded disk $D=D^{3}\hookrightarrow X\setminus S_{0}$ and let $S_1=\partial D$. Take an embedded arc $\gamma: [0,1]\to X$ that satisfies the following conditions:
\begin{itemize}
    \item $\gamma(0)= b$, $\gamma_{1}\in S_{1}$, $\gamma(0,1)\cap (S_{0}\cup S_{1})=\emptyset$;
    \item $\gamma$ intersects the interior of $D$ transversely at a single point $\gamma(\frac{1}{2})$. 
    \item Take any path in $\gamma':I\to D$ from $\gamma(1)$ to $\gamma(\frac{1}{2})$. Then $\widetilde{\gamma}:=\gamma|_{[\frac{1}{2},1]}\cdot \gamma'$ is a loop in $X$. We require that $\tilde{\gamma}$ is noncontractible.
    \item In the case $X=S^2\widetilde{\times}\Sigma$, we additionally require that the preimage of $\widetilde{\gamma}$ in $S^2\times \widetilde{\Sigma}$ has two components, denoted by $\widetilde{\gamma}_{1}$ and $\widetilde{\gamma}_{2}$.

\end{itemize}
We denote $S_0\cup \operatorname{im}(\gamma)\cup S_{1}$ by $\mathcal{B}_{\gamma}$ and call it a self-referential barbell. The trivializations of the normal bundle of $S_0$ and $S_1$ are unique up to homotopy. We pick any trivialization of the normal bundle of $\operatorname{im}(\gamma)$. After fixing these trivializations, we obtain a diffeomorphism between an $(S^2\times D^2)\natural (S^2\times D^2)$ and $\nu(\mathcal{B}_{\gamma})$ (a tubular neighborhood of $\mathcal{B}_{\gamma}$). Then we can push forward $\varphi$ using this diffeomorphism and extend it by the identity map on $X\setminus \nu(\mathcal{B}_{\gamma})$. This gives a diffeomorphism $f_{\gamma}: X\to X$ that is homotopic to the identity. We call it a \emph{self-referential barbell diffeomorphism}. 

Consider the surface $\Sigma_{1}=f_{\gamma}(
\Sigma_0)$. Take a tube along $\gamma$. Then $\Sigma_{1}$ is just the internal connected sum of $\Sigma_0$ and $S_1$ along this tube. We say that $\Sigma_{1}$ is obtained from $\Sigma_0$ by adding a self-referential disk along $\gamma$. 
Similarly, for any $m>1$, we let $\Sigma_{m}=f^{m}_{\gamma}(
\Sigma_0)$. Then $\Sigma_m$ is obtained from $\Sigma_0$ by adding $m$ self-referential disks along copies of $\gamma$.
These terminologies come from \cite{Gabaiself}.

\begin{proposition}\label{prop: Sigma_0 moved by self-ref diff}
For any $m>0$, the surface $\Sigma_m$ is not smoothly isotopic to $\Sigma_0$.    
\end{proposition}
\begin{proof} We first consider the case $X=S^2\times \Sigma$. In \cite{lin2025dax}
Xie, Zhang, and the first author defined the relative Dax invariant for any two embedded surfaces in $X$ that are both homotopic to $\Sigma_0$. This invariant takes value in $\mathbb{Z}\langle \mathcal{C}\rangle$, where $\mathcal{C}$ is the set of free homotopy classes of noncontractible loops in $X$. Consider the relative Dax invariant $\operatorname{Dax}(\Sigma_0,\Sigma_m)\in \mathbb{Z}\langle \mathcal{C}\rangle $. In \cite{Gabaiself}, Gabai computed the relative Dax invariant for self-referential disks in $(S^2\times D^2)\natural (S^1\times D^3)$. By naturality of the relative Dax invariant under codimension-0 embeddings (see \cite[Proof of Theorem 2.1]{lin2025dax}), we have 
\[
\operatorname{Dax}(\Sigma_0,\Sigma_m)=m([\widetilde{\gamma}]+[\widetilde{\gamma}^{-1}])\neq 0.
\] 
Therefore, $\Sigma_m$ is not smoothly isotopic to $\Sigma_0$.

We prove the case $X=S^2\widetilde{\times }\Sigma$ by passing to the double cover $p:S^2\times \widetilde{\Sigma}\to \Sigma$. Let $\widetilde{\Sigma}_{m}=p^{-1}(\Sigma_{m})$. Then $\widetilde{\Sigma}_{m}$ is obtained from $\widetilde{\Sigma}_0$ by adding $m$ self-referential disks along $\gamma_1$ and $m$ self-referential disks along $\gamma_2$. Here $\gamma_{1},\gamma_{2}$ denotes the two components of $p^{-1}(\gamma)$. Then 
\[
\operatorname{Dax}(\widetilde{\Sigma}_{0},\widetilde{\Sigma}_{m})=m([\widetilde{\gamma}_{1}]+[\widetilde{\gamma}^{-1}_{1}]+[\widetilde{\gamma}_{2}]+[\widetilde{\gamma}^{-1}_{2}])\neq 0\in \mathbb{Z}\langle \widetilde{\mathcal{C}}\rangle
\]
Here $\widetilde{\mathcal{C}}$ denotes the set of free homotopy classes of non-contractible loops in $S^2\times \widetilde{\Sigma}$. Hence $\widetilde{\Sigma}_{m}$ is not smoothly isotopic to $\widetilde{\Sigma}_0$. And  $\Sigma_0$ is not smoothly isotopic to $\Sigma_m$ as well. 
\end{proof}

By Proposition \ref{proposition symplectic suface standard} and Proposition \ref{prop: Sigma_0 moved by self-ref diff}, one immediately gets the following corollary, which implies that adding self-referential disks can not be made into a symplectic operation.
\begin{corollary}\label{cor: sigma_m not symplectic}
For any $m\neq 0$, the surface $\Sigma_{m}$ is not smoothly isotopic to a symplectic surface in $(X,\omega_{\mu})$.    
\end{corollary}

\begin{proof}[Proof of Theorem \ref{thm: main}] After rescaling, we may assume $a=[\omega_{\mu}]$. For $m\geq 0$, consider the symplectic form \[\omega_{m,\mu}:=(f^{m}_{\gamma})^*(\omega_{\mu})\in \mathcal{S}_{a}.\]
It suffices to show that for any $n< m$, the symplectic forms $\omega_{n,\mu}$ and $\omega_{m,\mu}$ are not isotopic. By pulling back both symplectic forms via the diffeomorphism $f^{-n}_{\gamma}$,  we may reduce to the case $n=0$. Now suppose $\omega_{m,\mu}$ is isotopic to $\omega_{\mu}$. Then by  Moser's argument, the diffeomorphism $f^{m}_{\gamma}$ is smoothly isotopic to some $f\in \Symp(X,\omega_{\mu})$. Hence $\Sigma_{m}=f^{m}(\Sigma_0)$ is smoothly isotopic to $f(\Sigma_0)$, which is a symplectic surface in $(X,\omega_{\mu})$. This is a contradiction to Corollary \ref{cor: sigma_m not symplectic}.
\end{proof}
\bibliographystyle{amsalpha}
\bibliography{biblio}

\end{document}